\documentclass[12pt]{article}
\usepackage{amsmath,amsthm,amssymb,geometry,lmodern,verbatim,enumerate,float,cite,tikz}
\usepackage{microtype}
\newtheorem{theorem}{Theorem}
\newtheorem{lemma}[theorem]{Lemma}

\newtheorem{proposition}[theorem]{Proposition}

\newtheorem*{assertion}{Assertion}

\DeclareMathOperator{\lpt}{lpt}
\DeclareMathOperator{\lct}{lct}

\renewcommand{\le}{\leqslant}
\renewcommand{\ge}{\geqslant}
\renewcommand{\leq}{\leqslant}
\renewcommand{\geq}{\geqslant}
\newcommand{\abs}[1]{\left\lvert#1\right\rvert}
\newcommand{\sst}[2]{\left\{#1\,:\,#2\right\}}
\newcommand\ifEven[3]{
\begingroup
\pgfmathsetmacro{\var}{#1}
\pgfmathparse{ifthenelse(mod(\var,2)==0,1,0)}
\ifdim\pgfmathresult pt= 1 pt
   #2
  \else
   #3
\fi
\endgroup
}

\begin{document}
\title{Transversals of Longest Paths and Cycles}
\author{Dieter Rautenbach$^1$ and Jean-S\'{e}bastien Sereni$^2$}
\date{}
\maketitle
\begin{center}
$^1$
Institut f\"{u}r Optimierung und Operations Research\\
Universit\"{a}t Ulm, Ulm, Germany, \texttt{dieter.rautenbach@uni-ulm.de}\\[3mm]
$^2$
Centre national de la recherche scientifique\\
LORIA, Vand{\oe}uvre-l\`{e}s-Nancy, France, \texttt{sereni@kam.mff.cuni.cz}
\end{center}

\begin{abstract}
Let $G$ be a graph of order $n$.
Let $\lpt(G)$ be the minimum cardinality of a set $X$ of vertices of $G$
such that $X$ intersects every longest path of $G$
and define $\lct(G)$ analogously for cycles instead of paths.
We prove that
\begin{itemize}
\item $\lpt(G)\leq \left\lceil\frac{n}{4}-\frac{n^{2/3}}{90}\right\rceil$, if $G$ is connected,
\item $\lct(G)\leq \left\lceil\frac{n}{3}-\frac{n^{2/3}}{36}\right\rceil$, if $G$ is $2$-connected, and
\item $\lpt(G)\leq 3$, if $G$ is a connected circular arc graph.
\end{itemize}
Our bound on $\lct(G)$ improves an earlier result of Thomassen
and
our bound for circular arc graphs relates
to an earlier statement of Balister \emph{et al.}~the argument of which contains a gap.
Furthermore, we prove upper bounds on $\lpt(G)$
for planar graphs and graphs of bounded tree-width.
\end{abstract}

{\small

}

\pagebreak

\section{Introduction}
It is well known that
every two longest paths in a connected graph
as well as
every two longest cycles in a $2$-connected graph intersect.
While these observations are easy exercises,
it is an open problem,
originating from a question posed by Gallai~\cite{ga},
to determine the largest value of $k$
such that for every connected graph and every $k$ longest paths in that graph,
there is a vertex that belongs to all of these $k$ paths.
The above remark along with examples constructed by Skupie\'{n}~\cite{sk}
ensure that $2\leq k\leq 6$.

We consider only simple, finite, and undirected graphs and use standard terminology.
For a graph $G$, we define $\mathcal{P}(G)$ to be the collection of all longest paths of $G$
and a \emph{longest path transversal} of $G$ to be a set of vertices that intersects every longest path of $G$.
Let $\lpt(G)$ be the minimum cardinality of a longest path transversal of $G$.
We define $\mathcal{C}(G)$, a \emph{longest cycle transversal}, and the parameter $\lct(G)$
analogously for cycles instead of paths.

The intersections of longest paths and cycles have been studied in detail
and Zamfirescu~\cite{za2} gave a short survey.
In the present paper we prove upper bounds on
$\lpt(G)$
and
$\lct(G)$.
Our bound on $\lct(G)$ for a $2$-connected graph $G$ improves an earlier
result of Thomassen~\cite{th}.
Balister \emph{et al.}~\cite{bagylesh} showed that
for every connected interval graph,
there is a vertex that belongs to every longest path.
Furthermore, their work
\cite{bagylesh} contains the statement that
for every connected circular arc graph,
there is a vertex that belongs to every longest path.
Unfortunately, we believe that the argument they provide has a gap.
We shall explain the approach of Balister \emph{et al.}, the problem with their argument,
and give a proof of a weaker result, specifically that every connected
circular arc graph contains a longest path transversal of order at most $3$.

\section{Results}
We start by proving a lemma that allows us to exploit the structure of
some particular matchings to find long paths and cycles.
\begin{lemma}\label{lemma1}
If $G=(P\cup Q)+M$
where
$P:u_1\ldots u_\tau$
and
$Q:v_1\ldots v_\tau$
are paths and $M$ is a matching of edges between $V(P)$ and $V(Q)$
that has a partition $M=M_1\cup \ldots\cup M_q$ such that
\begin{enumerate}[(a)]
\item $\abs{M_i}$ is either $1$ or even for $i\in [q]$ and
\item if $u_{i_1}v_{i_2}\in M_i$ and $u_{j_1}v_{j_2}\in M_j$ for $i,j\in [q]$,
then
\[
(j_1-i_1)(j_2-i_2)
\begin{cases}
< 0, & \text{if $i=j$ and}\\
> 0, & \text{if $i\not=j$,}
\end{cases}
\]
that is,
the edges in one of the sets $M_i$ are pairwise ``\emph{crossing}''
and
the edges in distinct sets $M_i$ are pairwise ``\emph{parallel}'',
\end{enumerate}
then $G$ contains a path between a vertex in $\{ u_1,v_1\}$ and a vertex in $\{ u_\tau,v_\tau\}$
of order at least $\tau+\abs{M}$.
\end{lemma}
\begin{proof}
If $i_0=1$, $i_{\abs{M}+1}=\tau$,
and
$u_{i_1},\ldots,u_{i_{\abs{M}}}$ with $1\leq i_1<\ldots<i_{\abs{M}}\leq \tau$ are the vertices of $P$
that are incident with edges in $M$,
then a subpath of $P$ of the form
$u_{i_j}\ldots u_{i_{j+1}}$ with odd/even $j$ is called an \emph{odd/even segment of $P$},
respectively.
Odd/even segments of $Q$ are defined analogously.

The odd segments of $P$, $M$, and the even segments of $Q$ define a path $P'$.
Similarly,
the even segments of $P$, $M$, and the odd segments of $Q$ define a path $Q'$.
See Figure~\ref{fig1} for an illustration.
Since $E(P')\cap E(Q')=M$ and $E(P')\cup E(Q')=E(P)\cup E(Q)\cup M$,
the longer of the two paths satisfies the desired properties.
\end{proof}

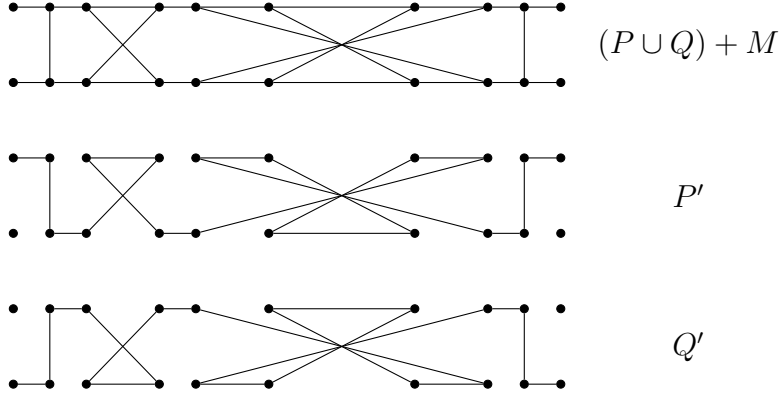
\begin{figure}[H]
\begin{center}
 \begin{tikzpicture}[x=.48cm,y=.5cm,vertex/.style={circle, draw=black, fill=black,
      inner sep=0mm, minimum size=3pt}]
\def\xList{1,2,4,5,7,11,13,14,15}
\foreach \y in {3,5} {
   \draw (0,\y) node[vertex] (0-\y) {};
      \foreach \x [count=\i] in \xList {%
         \pgfmathtruncatemacro\prevI{\i-1}
         \draw (\prevI-\y)--(\x,\y) node[vertex] (\i-\y) {};
  }
}
\foreach \y in {1,-1} {
   \draw (0,\y) node[vertex] (0-\y) {};
      \foreach \x [count=\i] in \xList {%
         \pgfmathtruncatemacro\prevI{\i-1}
         \draw (\x,\y) node[vertex] (\i-\y) {};
         \ifEven{\i}{\ifnum\y=-1{\draw (\prevI-\y)--(\i-\y);}\fi}{\ifnum\y=1{\draw (\prevI-\y)--(\i-\y);}\fi}
}
}
\foreach \y in {-3,-5} {
   \draw (0,\y) node[vertex] (0-\y) {};
      \foreach \x [count=\i] in \xList {%
         \pgfmathtruncatemacro\prevI{\i-1}
         \draw (\x,\y) node[vertex] (\i-\y) {};
         \ifEven{\i}{\ifnum\y=-3{\draw (\prevI-\y)--(\i-\y);}\fi}{\ifnum\y=-5{\draw (\prevI-\y)--(\i-\y);}\fi}
}
}
         \foreach \y in {-5,-1,3} {
         \pgfmathtruncatemacro\upp{\y+2}
         \draw (1-\y)--(1-\upp);
         \draw (2-\upp)--(3-\y);
         \draw (2-\y)--(3-\upp);
         \draw (4-\upp)--(7-\y);
         \draw (4-\y)--(7-\upp);
         \draw (5-\upp)--(6-\y);
         \draw (5-\y)--(6-\upp);
         \draw (8-\upp)--(8-\y);
         }
\draw (18.5,4) node {$(P\cup Q)+M$};
\draw (18.5,0) node {$P'$};
\draw (18.5,-4) node {$Q'$};
\end{tikzpicture}
\end{center}
\caption{The two paths $P'$ and $Q'$ for a set $M$ with
$\abs{M_1}=\abs{M_4}=1$, $\abs{M_2}=2$, and $\abs{M_3}=4$.}\label{fig1}
\end{figure}

\noindent We proceed to our first main result.
Note that in the proof of Theorem~\ref{theorem1},
as well as of Theorem~\ref{theorem2} below,
we did not try to minimize the factor of $n^{\frac{2}{3}}$.
The point of these two results is that
$\lpt(G)$ is strictly less than $n/4$
and
$\lct(G)$ is strictly less than $n/3$,
respectively.

\begin{theorem}\label{theorem1}
If $G$ is a connected graph of order $n$,
then $\lpt(G)\leq \left\lceil\frac{n}{4}-\frac{n^{2/3}}{90}\right\rceil$.
\end{theorem}
\begin{proof}
Let $G$ be a connected graph of order $n$.
Let
$\epsilon=\frac{1}{90}n^{-\frac{1}{3}}$
and $\tau=\left\lceil\left(\frac{1}{4}-\epsilon\right)n\right\rceil$.
For a contradiction, we assume that $\lpt(G)>\tau$.
Let $P:u_1\ldots u_\ell$ be a longest path of $G$.
Since $V(P)$ as well as every set of $n-\ell+1$ vertices of $G$ are longest
path transversals, we obtain
\begin{equation}\label{e1p}
\left(\frac{1}{4}-\epsilon\right)n\leq \tau<\ell <n-\tau+1\leq \left(\frac{3}{4}+\epsilon\right)n+1.
\end{equation}
Let $p=\left\lceil\frac{\ell-\tau}{2}\right\rceil$.
Since the set
$T=\sst{u_i}{p+1\leq i\leq p+\tau}$
is too small to be a longest path transversal of $G$,
there is a path $P':v_1\ldots v_\ell$ in $G-T$.
Since $G$ is connected,
the paths $P$ and $P'$ intersect.

If $V(P)\cap V(P')\subseteq \{ u_1,\ldots,u_p\}$,
then let $v_x=u_r\in V(P)\cap V(P')$ be such that $r$ is maximum.
We may assume that $x\geq \frac{\ell+1}{2}$.
Now $v_1\ldots v_xu_{r+1}\ldots u_{\ell}$ is a path of order at least
$x+\ell-p\geq \frac{\ell+1}{2}+\ell-\frac{\ell-\tau+1}{2}=\ell+\frac{\tau}{2}>\ell$,
which is a contradiction.
Hence $P$ and $P'$ intersect in a vertex in $\{ u_1,\ldots,u_p\}$
as well as a vertex in $\{ u_{p+\tau+1},\ldots,u_\ell\}$.
Let $v_x=u_r$ be in $V(P')\cap \{ u_1,\ldots,u_p\}$ such that $r$ is maximum
and $v_y=u_s$ be in $V(P')\cap \{ u_{p+\tau+1},\ldots,u_\ell\}$ such that $s$ is minimum.
We may assume that $x<y$.

Since $v_1\ldots v_xu_{r+1}\ldots u_{s-1}v_y\ldots v_\ell$ is a path of order at least
$\ell-(y-x-1)+\tau$,
we obtain $y-x-1\geq \tau$.
Since $u_{s-1}\ldots u_{r+1}v_x\ldots v_\ell$ is a path of order at least
$\tau+\ell-(x-1)$,
we obtain $x-1\geq \tau$.
Since $u_{r+1}\ldots u_{s-1}v_y\ldots v_1$ is a path of order at least
$\tau+y$,
we obtain $\ell-y\geq \tau$.

Choosing four vertex-disjoint paths
$A:a_1\ldots a_\tau$,
$B:b_1\ldots b_\tau$,
$C:c_1\ldots c_\tau$, and
$D:d_1\ldots d_\tau$
as subpaths of the four paths
$P'[\{ v_1,\ldots,v_{x-1}\}]$,
$P[\{ u_{r+1},\ldots,u_{s-1}\}]$,
$P'[\{ v_{x+1},\ldots,v_{y-1}\}]$, and
$P'[\{ v_{y+1},\ldots,v_\ell\}]$,
respectively,
we obtain the existence of two vertex-disjoint sets $X$ and $Y$
in $V(G)\setminus (V(A)\cup V(B)\cup V(C)\cup V(D))$
with $\abs{X \cup Y}\leq n-4\tau\leq 4\epsilon n$ such that
$X$ contains a path between some neighbors of any two of the vertices $a_1$, $b_1$, and $c_1$,
and
$Y$ contains a path between some neighbors of any two of the vertices $b_{\tau}$, $c_{\tau}$, and $d_1$.
See Figure~\ref{fig0} for an illustration.

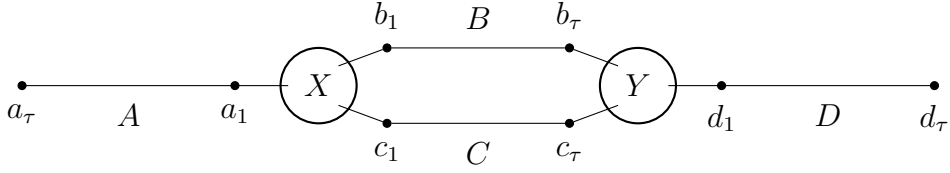
\begin{figure}[H]
\begin{center}
 \begin{tikzpicture}[vertex/.style={circle, draw=black, fill=black,
      inner sep=0mm, minimum size=3pt}]
\draw (-6,0) node[vertex] (L) {};
\node[below=3pt] at (L) {$a_{\tau}$};
\draw (-6,0)--(-3.2,0) node[vertex] (l) {} --(-2.5,0);
\node[below=3pt] at (l) {$a_1$};
\path (l)--(L) node[midway,below=3pt] {$A$};
\draw (-2.1,0) node[draw,circle,minimum size=10mm,thick] (x) {$X$};
\path (x)--++(45:3.7mm) coordinate (fl);
\draw (-1.2,.5) node[vertex] (cl) {};
\node[above=3pt] at (cl) {$b_1$};
\draw (fl)--(cl);
\path (x)--++(-45:3.7mm) coordinate (flb);
\draw (-1.2,-.5) node[vertex] (clb) {};
\node[below=3pt] at (clb) {$c_1$};
\draw (flb)--(clb);

\draw (6,0) node[vertex] (R) {};
\node[below=3pt] at (R) {$d_{\tau}$};
\draw (6,0)--(3.2,0) node[vertex] (r) {} --(2.5,0);
\node[below=3pt] at (r) {$d_1$};
\path (r)--(R) node[midway,below=3pt] {$D$};
\draw (2.1,0) node[draw,circle,minimum size=10mm,thick] (y) {$Y$};
\path (y)--++(135:3.7mm) coordinate (fr);
\draw (1.2,.5) node[vertex] (cr) {};
\node[above=3pt] at (cr) {$b_{\tau}$};
\path (cl)--(cr) node[midway,above=3pt] {$B$};
\draw (fr)--(cr);
\path (y)--++(-135:3.7mm) coordinate (frb);
\draw (1.2,-.5) node[vertex] (crb) {};
\node[below=3pt] at (crb) {$c_{\tau}$};
\draw (frb)--(crb);
\path (clb)--(crb) node[midway,below=3pt] {$C$};
\draw (cl)--(cr);
\draw (clb)--(crb);
\end{tikzpicture}
\end{center}
\caption{The paths $A$, $B$, $C$, and $D$ and the two sets $X$ and $Y$.}\label{fig0}
\end{figure}

\noindent If $a_ib_j$ is an edge of $G$ with $j\geq i$,
then
the path $a_\tau\ldots a_ib_j\ldots b_1$,
a path in $X$ between neighbors of $b_1$ and $c_1$,
the path $C$,
a path in $Y$ between neighbors of $c_\tau$ and $d_1$, and
the path $D$
form a path of order at least $3\tau+(j-i+1)+2$.
By~\eqref{e1p}, this implies that
\begin{equation}\label{e2p}
j-i\leq \lceil 4\epsilon n\rceil-3.
\end{equation}
Our goal is now to prove the existence of a vertex cover $T_{A,B}$ of small order
for the bipartite graph $G_{A,B}$ with bipartition $V(A)$ and $V(B)$
formed by the edges between these two sets.
First, note that if $4\epsilon n\le2$, then~\eqref{e2p} implies that
this bipartite graph is edgeless,
so it is enough to set $T_{A,B}=\emptyset$.

\noindent
Assume now that $4\epsilon n>2$.
Let $N$ be a maximum matching of $G_{A,B}$.
Let $I=\left[\left\lceil\frac{\tau}{2(\lceil 4\epsilon n\rceil-3)}\right\rceil\right]$.
For $i\in I$, let $N_i$ be the set of edges in $N$ that are incident with a vertex in
\[\sst{a_j}{2(\lceil 4\epsilon n\rceil-3)(i-1)+1\leq j\leq \min\{ \tau,2(\lceil 4\epsilon n\rceil-3)i\}}.\]
By~\eqref{e2p},
if $a_{i_1}b_{i_2}\in N_i$ and $a_{j_1}b_{j_2}\in N_j$ with $j-i\geq 2$,
then $(j_1-i_1)(j_2-i_2)>0$, that is, the two edges are parallel in the sense of Lemma~\ref{lemma1}.
Without loss of generality, we may assume that
$\bigcup_{\text{$i\in I:i$ odd}}N_i$ contains at least half the edges of $N$.
Since $N$ is a matching, $\abs{N_i}\leq 2(\lceil 4\epsilon n\rceil-3)$ for every $i\in I$.
Since permutation graphs are perfect~\cite[Chapter 7]{go},
each $N_i$ contains a set of at least $\sqrt{\abs{N_i}}\geq \frac{\abs{N_i}}{\sqrt{2(\lceil 4\epsilon n\rceil-3)}}$ edges
that are either all pairwise parallel
or all pairwise crossing in the sense of Lemma~\ref{lemma1}.
This implies that $N$ contains a subset $M_0$
that satisfies condition (b) from Lemma~\ref{lemma1} with
$\abs{M_0}\geq
\sum_{\text{$i\in I:i$ odd}}\frac{\abs{N_i}}{\sqrt{2(\lceil 4\epsilon n\rceil-3)}}\geq
\frac{1}{2\sqrt{2(\lceil 4\epsilon n\rceil-3)}}\abs{N}$.
By removing a set of at most $\abs{M_0}/3$ edges from $M_0$,
we obtain a matching $M$ that satisfies both conditions from Lemma~\ref{lemma1} with
$\abs{M}\geq \frac{2}{3}\abs{M_0}\geq \frac{1}{3\sqrt{2(\lceil 4\epsilon n\rceil-3)}}\abs{N}$.

If $N$ has order at least $(4\epsilon n-1)3\sqrt{2(\lceil 4\epsilon n\rceil-3)}$,
then $N$ contains a matching $M$ as in Lemma~\ref{lemma1} of order at least $4\epsilon n-1$.
Thus by Lemma~\ref{lemma1}, the graph
$(A\cup B)+M$ contains a path $Q$ between $\{a_1,b_1\}$ and $\{a_\tau,b_\tau\}$
of order at least $\tau+\abs{M}$.
Now
the path $Q$,
a path in $X$ between neighbors of a vertex in $\{ a_1,b_1\}$ and $c_1$,
the path $C$,
a path in $Y$ between neighbors of $c_\tau$ and $d_1$, and
the path $D$
form a path of order at least $3\tau+\abs{M}+2\geq \left(\frac{3}{4}+\epsilon\right)n+1$,
which contradicts~\eqref{e1p}.
Hence the bipartite graph $G_{A,B}$
has no matching of order at least $(4\epsilon n-1)3\sqrt{2(\lceil 4\epsilon n\rceil-3)}$.
Now K\"{o}nig's theorem~\cite{ko} implies that
$G_{A,B}$ has a vertex cover $T_{A,B}$
of order less than $(4\epsilon n-1)3\sqrt{2(\lceil 4\epsilon n\rceil-3)}$.

Similar arguments yield that for every two distinct paths $Q,R\in \{ A,B,C,D\}$,
the bipartite graph $G_{Q,R}$ with bipartition $V(Q)$ and $V(R)$
formed by the edges between these two sets
has a vertex cover $T_{Q,R}$
of order less than $(4\epsilon n-1)3\sqrt{2(\lceil 4\epsilon n\rceil-3)}$ if
$4\epsilon n>2$ and of order $0$ otherwise.
(If, for instance, $a_id_j$ is an edge of $G$ with $j\geq i$,
then
the path $a_\tau\ldots a_id_j\ldots d_1$,
a path in $Y$ between neighbors of $d_1$ and $b_\tau$,
the path $B$,
a path in $X$ between neighbors of $b_1$ and $c_1$, and
the path $C$
again form a path of order at least $3\tau+(j-i+1)+2$
and we can argue as above.)

Let $T'=X\cup Y\cup\bigcup_{{\left\{Q,R \right\}\in\binom{\left\{A,B,C,D\right\}}{2}}}T_{Q,R}$.
Since every component of $G-T'$ has order at most $\tau$,
the set $T'$ is a longest path transversal of $G$ of order less than
$4\epsilon n+6(4\epsilon n-1)3\sqrt{2(\lceil 4\epsilon n\rceil-3)}$ if
$4\epsilon n>2$ and of order at most $4\epsilon n$ otherwise.
For
$\epsilon=\frac{1}{90}n^{-\frac{1}{3}}$, it follows that
$\abs{T'}\le\left(\frac{1}{4}-\epsilon\right)n$,
which yields the final contradiction.
\end{proof}

\noindent For the fractional version of the longest path transversal problem, a much stronger result is possible.
In fact, for every connected graph $G$,
there is a function $t\colon V(G)\to [0,1]$ such that
\begin{align*}
\sum_{u\in V(G)}t(u)&\leq \sqrt{n}\quad\text{and}\\
\sum_{u\in V(P)}t(u)&\geq 1\quad\text{for every $P\in \mathcal{P}(G)$.}\\
\end{align*}
Indeed, if the largest order of the paths in $\mathcal{P}(G)$ is at most $\sqrt{n}$,
then let $t$ be the characteristic function of $V(P)$ for some $P\in \mathcal{P}(G)$,
otherwise let $t$ be the constant function of value $\frac{1}{\sqrt{n}}$.

Confirming a conjecture by Zamfirescu~\cite{za1},
Thomassen~\cite{th} proved that $\lct(G)\leq \left\lceil\frac{\abs{V(G)}}{3}\right\rceil$ for every graph $G$,
which is best possible for the class of connected graphs
in view of a disjoint union of cycles of length $3$ to which bridges are added.
For $2$-connected graphs, though, this bound can be improved as follows.
\begin{theorem}\label{theorem2}
If $G$ is a $2$-connected graph of order $n$,
then $\lct(G)\leq \left\lceil\frac{n}{3}-\frac{n^{2/3}}{36}\right\rceil$.
\end{theorem}
\begin{proof}
Let $G$ be a $2$-connected graph of order $n$.
Let
$\epsilon=\frac{1}{36}n^{-\frac{1}{3}}$
and $\tau=\left\lceil\left(\frac{1}{3}-\epsilon\right)n\right\rceil$.
For a contradiction, we assume that $\lct(G)>\tau$.
Let $C:u_0\ldots u_{\ell-1} u_0$ be a longest cycle of $G$.
Since $V(C)$ as well as every set of $n-\ell+1$ vertices of $G$ are longest
cycle transversals,
we obtain
\begin{equation}\label{e1}
\left(\frac{1}{3}-\epsilon\right)n\leq \tau <\ell < n-\tau+1\leq \left(\frac{2}{3}+\epsilon\right)n+1.
\end{equation}
Since the set $T=\{ u_0,\ldots,u_{\tau-1}\}$ is too small to be a longest cycle transversal of $G$,
there is a cycle $C':v_0\ldots v_{\ell-1} v_0$ in $G-T$.
Since $G$ is $2$-connected,
the cycles $C$ and $C'$ intersect in at least two vertices.
We may assume that $v_0=u_r$ is the first and
$v_k=u_s$ is the last common vertex of $C$ and $C'$
following the path $C-T$ from $u_\tau$ to $u_{\ell-1}$, that is $r<s$.

Since $v_0\ldots v_ku_{s+1}\ldots u_{\ell-1}u_0\ldots u_{r-1}$ is a cycle of length at least $k+1+\tau$,
we obtain $\ell-k-1\geq \tau$.
Since $v_{k+1}\ldots v_{\ell-1}u_r\ldots u_0u_{\ell-1}\ldots u_s$ is a cycle of length at least $\ell-(k-1)+\tau$,
we obtain $k-1\geq \tau$.

Choosing three vertex-disjoint paths
$P:x_1\ldots x_\tau$,
$Q:y_1\ldots y_\tau$, and
$R:z_1\ldots z_\tau$
as subpaths of the three paths
$C[T]$,
$C'[\{ v_1,\ldots,v_{k-1}\}]$, and
$C'[\{ v_{k+1},\ldots,v_{\ell-1}\}]$,
respectively,
we obtain the existence of two vertex-disjoint sets $X$ and $Y$ in $V(G)\setminus (V(P)\cup V(Q)\cup V(R))$
with $\abs{X\cup Y}\leq n-3\tau\leq 3\epsilon n$ such that
$X$ contains a path between some neighbors of every two of the vertices $x_1$, $y_1$, and $z_1$,
and
$Y$ contains a path between some neighbors of every two of the vertices $x_\tau$, $y_\tau$, and $z_\tau$.
See Figure~\ref{fig3} for an illustration.

\begin{figure}[H]
\begin{center}
 \begin{tikzpicture}[scale=.6,vertex/.style={circle, draw=black, fill=black,
      inner sep=0mm, minimum size=3pt}]
\draw (-4,0) node[vertex] (lm) {};
\node[below=3pt] at (lm) {$y_{1}$};
\draw (4,0) node[vertex] (rm) {};
\node[below=3pt] at (rm) {$y_{\tau}$};
\draw (lm)--(rm) node[midway,below=3pt] {$Q$};
\draw (-4,1.5) node[vertex] (lt) {};
\node[below=3pt] at (lt) {$x_{1}$};
\draw (4,1.5) node[vertex] (rt) {};
\node[below=3pt] at (rt) {$x_{\tau}$};
\draw (lt)--(rt) node[midway,below=3pt] {$P$};
\draw (-4,-1.5) node[vertex] (lb) {};
\node[below=3pt] at (lb) {$z_{1}$};
\draw (4,-1.5) node[vertex] (rb) {};
\node[below=3pt] at (rb) {$z_{\tau}$};
\draw (lb)--(rb) node[midway,below=3pt] {$R$};
\draw (-5.5,0) node[draw,circle,minimum size=10mm,thick] (x) {$X$};
\path (x)--++(45:5.5mm) coordinate (ft);
\draw (ft)--(lt);
\path (x)--++(0:5mm) coordinate (fm);
\draw (fm)--(lm);
\path (x)--++(-45:5.5mm) coordinate (fb);
\draw (fb)--(lb);
\draw (5.5,0) node[draw,circle,minimum size=10mm,thick] (y) {$Y$};
\path (y)--++(135:5.5mm) coordinate (Rt);
\draw (Rt)--(rt);
\path (y)--++(180:5mm) coordinate (Rm);
\draw (Rm)--(rm);
\path (y)--++(-135:5.5mm) coordinate (Rb);
\draw (Rb)--(rb);
\end{tikzpicture}
\end{center}
\caption{The paths $P$, $Q$, and $R$ and the two sets $X$ and $Y$.}\label{fig3}
\end{figure}
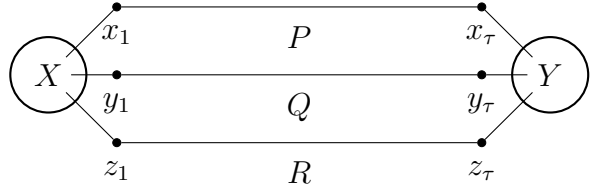

\noindent If $x_iy_j$ is an edge of $G$ with $j\geq i$, then
the path $y_1\ldots y_jx_i\ldots x_\tau$,
a path in $Y$ between neighbors of $x_\tau$ and $z_\tau$
the path $R$, and
a path in $X$ between neighbors of $y_1$ and $z_1$
form a cycle of length at least $2\tau+(j-i+1)+2$.
By~\eqref{e1}, this implies that
$j-i\leq \lceil 3\epsilon n\rceil-3$.
Using Lemma~\ref{lemma1} as in the proof of Theorem~\ref{theorem1},
we infer that for every two distinct $A,B\in\{ P,Q,R\}$,
the bipartite graph $G_{A,B}$ with bipartition $V(A)$ and $V(B)$
formed by the edges between these two sets
has a vertex cover $T_{A,B}$
of order less than
$(3\epsilon n-1)3\sqrt{2(\lceil 3\epsilon n\rceil-3)}$ if $3\epsilon n>2$
and of order
$0$ otherwise.
Since every component of $G-(X\cup Y\cup T_{P,Q}\cup T_{P,R}\cup T_{Q,R})$
has order at most $\tau$,
the set $T'=X\cup Y\cup T_{P,Q}\cup T_{P,R}\cup T_{Q,R}$
is a longest cycle transversal of $G$ of order less than
$3\epsilon n+3(3\epsilon n-1)3\sqrt{2(\lceil 3\epsilon n\rceil-3)}$ if $3\epsilon n>2$
and of order at most
$3\epsilon n$ otherwise.
For
$\epsilon=\frac{1}{36}n^{-\frac{1}{3}}$, it follows that
$\abs{T'}\le\left(\frac{1}{3}-\epsilon\right)n$,
which yields the final contradiction.
\end{proof}

\noindent Since every two longest paths of a connected graph $G$ intersect, it follows
that $\lpt(G)\leq \left\lceil\frac{\abs{\mathcal{P}(G)}}{2}\right\rceil$.
Similarly,
if every $k$ longest paths of a connected graph $G$ would intersect for some $k\geq 3$,
then it would follow that $\lpt(G)\leq \left\lceil\frac{\abs{\mathcal{P}(G)}}{k}\right\rceil$.
The next result shows how to decrease the
multiplicative constant $1/2$ in the former bound at the cost of adding a
square-root proportion of the total number of vertices in the graph.
\begin{proposition}\label{theorem3}
If $G$ is a connected graph and $\alpha\geq 2$, then
\[\lpt(G)\leq \frac{\abs{\mathcal{P}(G)}}{\alpha}+\sqrt{\alpha \abs{V(G)}}.\]
\end{proposition}
\begin{proof}
We proceed by induction on the order $n$ of $G$, the statement being true if
$n=1$.
Let $n\ge2$ and assume that the statement holds for all connected graphs of
order less than $n$. Let $G$ be a connected graph of order $n$ and let $\ell$
be the order of the longest paths in $G$.

We may assume that $\abs{\mathcal{P}(G)}>\sqrt{\alpha n}$ since otherwise
we obtain a longest path transversal of the desired size by picking one
vertex in each longest path of $G$. Next,
since the vertex set of a longest path in $G$ is a longest path
transversal, we may also assume that $\ell>\sqrt{\alpha n}$.

For a vertex $v\in V(G)$, let $p_v$ be the number of paths in $\mathcal{P}(G)$
that contain $v$. We may assume that $p_v<\alpha$ for every vertex $v\in
V(G)$. Indeed, suppose that $v$ is a vertex such that $p_v\ge\alpha$.
In particular, $p_v\ge1$.
If the set $\{v\}$ is a longest path
transversal of $G$, then $G$ satisfies the desired property.
Otherwise let $G'=G-v$ and note that $G'$ contains
a path of order $\ell$. Furthermore,
$\abs{\mathcal{P}(G')}=\abs{\mathcal{P}(G)}-p_v\le\abs{\mathcal{P}(G)}-\alpha$.
Note that all paths of order $\ell$ in $G'$ must belong to
the same component of $G'$, since
every two longest paths intersect. Let $C$ be this component; thus
$\mathcal{P}(C)=\mathcal{P}(G')$.
The induction hypothesis applied to $C$ yields that
$\lpt(G')\leq \frac{\abs{\mathcal{P}(G)}}{\alpha}-1+\sqrt{\alpha (n-1)}$.
As $\lpt(G)\le\lpt(G')+1$, we deduce that
$\lpt(G)\le\frac{\abs{\mathcal{P}(G)}}{\alpha}+\sqrt{\alpha n}$.

We now consider the number $N$ of pairs $(v,P)$ such that $P\in\mathcal{P}(G)$ and
$v\in V(P)$. One the one hand, since $N=\sum_{v\in V(G)}p_v$, we deduce from the previous
observations that $N<\alpha n$. On the other hand, since
$N=\sum_{P\in\mathcal{P}(G)}\abs{V(P)}=\ell\abs{\mathcal{P}(G)}$, the previous observations also imply that
$N>\alpha n$. This contradiction concludes the proof.
\end{proof}

\noindent The minimum sizes of transversals of longest paths can be bounded in classes of graphs
with small separators, such as planar graphs and graphs of bounded
tree-width. As before, no effort is made to minimize the constant
multiplicative factors appearing in the next two results.
\begin{proposition}\label{theorem4}
If $G$ is a connected planar graph of order at least $2$, then
\[\lpt(G)\le9\sqrt{\abs{V(G)}}\log \abs{V(G)}.\]
\end{proposition}
\begin{proof}
We proceed by induction on the order $n$ of $G$, the result being true if $n=2$.
Let $n\ge3$ and assume that the statement holds for all connected planar
graphs of order at least $2$ and
less than $n$. Let $G$ be a connected planar graph of order $n$ and let $\ell$ be the order of the
longest paths in $G$.
In particular, $\ell\ge2$.
Since $G$ is planar, the separator theorem of Lipton and Tarjan ensures that
$G$ contains a set $X$ of order at most $2\sqrt{2}\sqrt{n}$ such that every component
of $G-X$ has order at most $2n/3$.

If $X$ is a longest path transversal of $G$, then $G$ satisfies the desired
property. Otherwise, $G-X$ contains a path of order $\ell$. Note that, since
every two longest paths of $G$ intersect,
all paths of order $\ell$ in $G-X$ must be contained in the same component of $G-X$, which we call $C$.
Moreover, the order of $C$ is at most $2n/3$ and at least $\ell$, so the induction hypothesis
implies that $C$ has a longest path transversal $X'$ of order at most
$9\sqrt{2n/3}\log(2n/3)$. Therefore, $X\cup X'$ is a longest path transversal
of $G$ of order at most
\begin{align*}
   2\sqrt{2}\sqrt{n}+9\sqrt{2n/3}\log(2n/3)&=9\sqrt{2n/3}\log
   n+\sqrt{n}\cdot\left(2\sqrt{2}-9\sqrt{2/3}\log(3/2)\right)\\
   &\le9\sqrt{n}\log n
\end{align*}
since $9\sqrt{2/3}\log(3/2)>2\sqrt{2}$. This concludes the proof.
\end{proof}

\noindent An anlaguous statement is true for graphs of bounded tree-width. Indeed,
if $G$ is a graph with tree-width at most $k$, then there is a set $X$ of
vertices of $G$ of order at most $k+1$ such that every component of $G-X$
has order at most $\abs{V(G)}/2$. Consequently, an inductive reasoning similar
to that made in the proof of Proposition~\ref{theorem4} yields the following
statement.
\begin{proposition}\label{theorem5}
If $G$ is a connected graph of tree-width at most $k$ and order at least $2$, then
\[\lpt(G)\le3k\log \abs{V(G)}.\]
\end{proposition}

\bigskip

\noindent We proceed to circular arc graphs.
We explain the approach of Balister \emph{et al.}~\cite{bagylesh},
the problem with their argument,
and prove the following weaker result.

\begin{theorem}\label{theorem6}
Let $G$ be a circular-arc graph.

If $G$ is connected, then $\lpt(G)\leq 3$,
and if $G$ is $2$-connected, then $\lct(G)\leq 3$.
\end{theorem}
\noindent Let $G$ be a connected circular arc graph.
Let $C$ be a circle and let $\mathcal{F}$ be a collection of open arcs of $C$
such that $G$ is the intersection graph of $\mathcal{F}$.
In view of the result for interval graphs mentioned in the introduction,
we may assume that $C\subseteq \bigcup_{A\in \mathcal{F}}A$.
Furthermore, we may assume that all endpoints of arcs in $\mathcal{F}$ are distinct.

Balister \emph{et al.}~\cite{bagylesh} consider a collection
$ \mathcal{K}=\{ K_0,\ldots,K_{n-1}\}$ of arcs in $\mathcal{F}$ such that
\begin{enumerate}
\item[(0)] $C\subseteq \bigcup_{A\in \mathcal{K}}A$,
\item[(1)] $n$ is minimal, and
\item[(2)] each $K_i$ is maximal, that is, no arc in $\mathcal{F}$ properly contains an arc in $\mathcal{K}$.
\end{enumerate}
They may assume that $n\geq 2$,
because otherwise, $G$ has a universal vertex that belongs to every longest path or cycle.
We consider the indices of the arcs in $\mathcal{K}$ as elements of $\mathbf{Z}_n$, that is, modulo $n$.

A \emph{chain of order $\ell$ in $\mathcal{F}$} is a sequence $\mathcal{P}:A_1\ldots A_\ell$ of distinct arcs in $\mathcal{F}$
such that $A_i\cap A_{i+1}\not=\emptyset$ for $i\in [\ell-1]$.
The chain $\mathcal{P}$ is \emph{closed}, if $A_\ell\cap A_1\neq\emptyset$.
Thus chains and closed chains in $\mathcal{F}$
correspond to paths and cycles in $G$.
For a chain $\mathcal{P}:A_1\ldots A_\ell$ in $\mathcal{F}$,
let $\mathcal{K}(\mathcal{P})=\{ A_1,\ldots,A_\ell\}\cap \mathcal{K}$.

If $\mathcal{P}:A_1\ldots A_\ell$ is a chain in $\mathcal{F}$ of largest order, then
Balister \emph{et al.}~\cite[Lemma 3.1]{bagylesh} proved that
$\mathcal{K}(\mathcal{P})$ is of the form $\sst{K_i}{i\in I}$ for some
contiguous and non-empty subset $I$ of $\mathbf{Z}_n$.
Their argument
actually yields the same statement for closed chains,
that is,
if $\mathcal{C}$ is a closed chain in $\mathcal{F}$ of largest order, then
$\mathcal{K}(\mathcal{C})$ is of the form $\sst{K_i}{i\in J}$
for some contiguous and non-empty subset $J$ of $\mathbf{Z}_n$.

\bigskip

\noindent In the proof of their main result~\cite[Theorem 3.3]{bagylesh}
--- stating that $\lpt(G)=1$ ---
Balister \emph{et al.} choose a chain $\mathcal{P}$ in $\mathcal{F}$ of largest order
such that $\mathcal{K}(\mathcal{P})$ has minimum order.
They let $\mathcal{K}(\mathcal{P})$ be $\{ K_{a+1},\ldots,K_{b-1}\}$
and assert that $K_{b-1}$ belongs to $\mathcal{K}(\mathcal{Q})$
for every chain $\mathcal{Q}$ in $\mathcal{F}$ of largest order,
that is, the vertex of $G$ corresponding to the arc $K_{b-1}$ would belong to every longest path of $G$.

For a contradiction, they consider a chain $\mathcal{Q}$ in $\mathcal{F}$ of largest order
such that $K_{b-1}\not\in \mathcal{K}(\mathcal{Q})$.
They set $\mathcal{K}(\mathcal{Q})=\{ K_{\ell+1},\ldots,K_{m-1}\}$.
They deduce from the choice of $\mathcal{P}$ that
$K_{\ell+1}\in \mathcal{K}(\mathcal{Q})\setminus \mathcal{K}(\mathcal{P})$
since
$K_{b-1}\in \mathcal{K}(\mathcal{P})\setminus \mathcal{K}(\mathcal{Q})$.
Using their Lemma 3.2~\cite{bagylesh},
they reorder the arcs in the chains $\mathcal{P}$ and $\mathcal{Q}$ and obtain
chains $\mathcal{P}^*$ and $\mathcal{Q}^*$
containing the same arcs as $\mathcal{P}$ and $\mathcal{Q}$ in a possibly different order, respectively.
They split these chains at $K_{b-1}$ and $K_{\ell+1}$ writing them as
$\mathcal{P^*}:\mathcal{P}_1K_{b-1}\mathcal{P}_2$
and
$\mathcal{Q^*}:\mathcal{Q}_1K_{\ell+1}\mathcal{P}_2$, respectively.

Their core statement is that
$\mathcal{C}_1:\mathcal{P}_1K_{b-1}\mathcal{R}K_{\ell+1}\mathcal{Q}_1^r$
and
$\mathcal{C}_2:\mathcal{P}_2^rK_{b-1}\mathcal{R}K_{\ell+1}\mathcal{Q}_2$
are chains
that satisfy the inequality
$\abs{\mathcal{C}_1}+\abs{\mathcal{C}_2}\geq 2+\abs{\mathcal{P}}+\abs{\mathcal{Q}}$,
where
$\mathcal{R}$ is the possibly empty chain $K_{b}\ldots K_{\ell}$
and the exponent ``$r$'' means reversal.
In order to prove this statement,
they have to show that no arc appears twice in these sequences.
They give details only for $\mathcal{C}_1$.
Their argument that $\mathcal{C}_1$ is a chain
heavily relies on the properties of the reordered chains $\mathcal{P}^*$ and $\mathcal{Q}^*$
guaranteed by their Lemma 3.2.
In the proof of Lemma 3.2 these properties are established
by iteratively shifting within $\mathcal{P}$
the arc $K_{b-1}$ to the beginning of $\mathcal{P}$
and, similarly,
by iteratively shifting within $\mathcal{Q}$
the arc $K_{\ell+1}$ to the beginning of $\mathcal{Q}$.
After proving that $\mathcal{C}_1$ is indeed a chain,
they say that the same type of argument shows that $\mathcal{C}_2$ is a chain as well.

This is the gap in their argument.

In order to use the same type of argument for $\mathcal{C}_2$,
they would need reversed versions of the properties guaranteed by Lemma 3.2,
that is, in order to establish these properties they would have to
iteratively shift within $\mathcal{P}$ the arc $K_{b-1}$ to the end of $\mathcal{P}$
and, similarly,
to iteratively shift within $\mathcal{Q}$ the arc $K_{\ell+1}$ to the end of $\mathcal{Q}$.
This may easily result in reorderings that are distinct from $\mathcal{P}^*$ and $\mathcal{Q}^*$.
In view of this asymmetry, the suitably adapted chain $\mathcal{C}_2$,
which would use the different reorderings of $\mathcal{P}$ and $\mathcal{Q}$,
need not satisfy the crucial inequality
$\abs{\mathcal{C}_1}+\abs{\mathcal{C}_2}\geq 2+\abs{\mathcal{P}}+\abs{\mathcal{Q}}$
and the argument breaks down.

We proceed to the proof of our Theorem~\ref{theorem6}.

\begin{proof}[Proof of Theorem~\ref{theorem6}]
Let $G$ be a connected circular arc graph.
We choose $C$, $\mathcal{F}$, and $\mathcal{K}$ exactly as above and we start
by proving the following statement.

\begin{assertion}
If $\mathcal{P}$ and $\mathcal{Q}$ are chains of largest order in $\mathcal{F}$ such that
\begin{align*}
\mathcal{K}(\mathcal{P})&=\{ K_{a+1},\ldots,K_{b-1}\}=\sst{K_i}{i\in I(\mathcal{P})}\quad\text{and}\\
\mathcal{K}(\mathcal{Q})&=\{ K_{\ell+1},\ldots,K_{m-1}\}=\sst{K_i}{i\in I(\mathcal{Q})}
\end{align*}
are disjoint,
then $a=m-1$ or $b=\ell+1$,
that is, the subsets $I(\mathcal{P})$ and $I(\mathcal{Q})$ of $\mathcal{Z}_n$
are contiguous.
\end{assertion}
To establish this assertion, assume on the contrary that
$a\not=m-1$ and $b\not=\ell+1$.
Select a set $S(\mathcal{P})$ of points of $C$ such that
$S(\mathcal{P})$ contains a point in the intersection of every two consecutive arcs of $\mathcal{P}$.
Define $S(\mathcal{Q})$ analogously.
If $K_a$ or $K_b$ would intersect $S(\mathcal{P})$ or $S(\mathcal{Q})$,
then $K_a$ or $K_b$ could be inserted into $\mathcal{P}$ or $\mathcal{Q}$,
respectively,
contradicting the assumption that these chains are of largest order.
If $S(\mathcal{P})$ or $S(\mathcal{Q})$ would intersect both arcs of $C\setminus (K_a\cup K_b)$,
then some arc of $\mathcal{P}$ or $\mathcal{Q}$ would properly contain $K_a$ or $K_b$,
which yields a contradiction to the condition (2) in the choice of $\mathcal{K}$.
Since $\mathcal{K}(\mathcal{P})$ and $\mathcal{K}(\mathcal{Q})$ are disjoint,
the sets $S(\mathcal{P})$ and $S(\mathcal{Q})$
are contained in different of the two arcs of $C\setminus (K_a\cup K_b)$.
Since $G$ is connected, $\mathcal{P}$ and $\mathcal{Q}$ have a common arc $A$.
This arc $A$ intersects $S(\mathcal{P})$ as well as $S(\mathcal{Q})$,
that is, it intersects both arcs of $C\setminus (K_a\cup K_b)$.
Hence either $K_a$ or $K_b$ is properly contained in $A$,
which again yields a contradiction to the condition (2) in the choice of $\mathcal{K}$.
This conludes the proof of the assertion.

\bigskip

\noindent Again let $\mathcal{P}$ be a chain in $\mathcal{F}$ of largest order
and, subject to this, such that $\mathcal{K}(\mathcal{P})$ has minimum order.
Let $\mathcal{K}(\mathcal{P})=\{ K_{a+1},\ldots,K_{b-1}\}$.

In view of the desired statement,
we may assume that $\mathcal{F}$ contains a chain $\mathcal{Q}$ of largest order
such that $K_{a+1},K_{b-1}\not\in \mathcal{K}(\mathcal{Q})$.
Among all such chains,
we assume that $\mathcal{Q}$ is chosen such that $\mathcal{K}(\mathcal{Q})$ has minimum order.
Let $\mathcal{K}(\mathcal{Q})=\{ K_{\ell+1},\ldots,K_{m-1}\}$.
By the choice of $\mathcal{P}$,
the sets
$\mathcal{K}(\mathcal{P})$
and
$\mathcal{K}(\mathcal{Q})$
are disjoint.
By the assertion,
we may assume that $b=\ell+1$.

In view of the desired statement,
we may assume that $\mathcal{F}$ contains a chain $\mathcal{R}$ of largest order
such that $K_{a+1},K_{b-1},K_{m-1}\not\in \mathcal{K}(\mathcal{R})$.
Among all such chains,
we assume that $\mathcal{R}$ is chosen such that $\mathcal{K}(\mathcal{R})$ has minimum order.
Let $\mathcal{K}(\mathcal{R})=\{ K_{p+1},\ldots,K_{q-1}\}$.
By the choice of $\mathcal{P}$ and $\mathcal{Q}$,
the sets
$\mathcal{K}(\mathcal{P})\cup \mathcal{K}(\mathcal{Q})$
and
$\mathcal{K}(\mathcal{R})$
are disjoint.
Applying the assertion
to $\mathcal{P}$ and $\mathcal{R}$
as well as
to $\mathcal{Q}$ and $\mathcal{R}$,
we obtain $p=m-1$ and $q=a+1$,
that is,
$\mathcal{K}(\mathcal{P})\cup \mathcal{K}(\mathcal{Q})\cup \mathcal{K}(\mathcal{R})$
is a partition of $\mathcal{K}$.

In view of the desired statement,
we may assume that $\mathcal{F}$ contains a chain $\mathcal{S}$ of largest order
such that $K_{a+1},K_{\ell+1},K_{p+1}\not\in \mathcal{K}(\mathcal{S})$.
We deduce from the choice of $\mathcal{P}$ that
$\mathcal{K}(\mathcal{S})$ has at least as many elements as $\mathcal{K}(\mathcal{P})$.
This implies that $\mathcal{K}(\mathcal{S})$ is disjoint from $\mathcal{K}(\mathcal{P})$.
Now, by the choice of $\mathcal{Q}$, this implies that
$\mathcal{K}(\mathcal{S})$ has at least as many elements as $\mathcal{K}(\mathcal{Q})$.
This in turn implies that
the set $\mathcal{K}(\mathcal{S})$ is disjoint from $\mathcal{K}(\mathcal{P})\cup \mathcal{K}(\mathcal{Q})$.
Finally, by the choice of $\mathcal{R}$, this implies that
$\mathcal{K}(\mathcal{S})$ has at least as many elements as $\mathcal{K}(\mathcal{R})$.
This in turn implies that
the set $\mathcal{K}(\mathcal{S})$ equals $\mathcal{K}(\mathcal{R})$,
that is, $\mathcal{K}(\mathcal{S})$ contains $K_{p+1}$,
which is a contradiction.

This completes the proof that $\lpt(G)$ is at most $3$.

\bigskip

\noindent From now on we assume that $G$ is a $2$-connected circular arc graph,
that is, every two longest cycles in $G$ --- closed chains of largest order in $\mathcal{F}$ --- intersect.
It is straightforward to see that the assertion also applies to closed chains instead of chains.
Arguing exactly as above for closed chains in $\mathcal{F}$
instead of chains in $\mathcal{F}$ implies that $\lct(G)\leq 3$.
\end{proof}

\noindent \textbf{Acknowledgement}\\
This research was initiated during a visit of the first author to LORIA.
It was supported by the French \emph{Agence nationale de la recherche} under reference
\textsc{anr 10 jcjc 0204 01}.

\end{document}